\documentclass[a4paper, 12pt, DIV=11]{scrartcl}

\title{\LARGE Partitioning edge-coloured hypergraphs into few  monochromatic tight cycles\thanks{Published in SIAM J. Discrete Math. \textbf{34} (2020), no. 2, 1460--1471, DOI: \href{https://doi.org/10.1137/19M1269786}{10.1137/19M1269786}, \copyright~2020. This manuscript version is made available under the \href{http://creativecommons.org/licenses/by-nc-nd/4.0/}{CC-BY-NC-ND 4.0} license.}}
\author{Sebasti\'an Bustamante \and Jan Corsten \and N\'ora Frankl \and Alexey Pokrovskiy \and Jozef Skokan}

\usepackage[utf8]{inputenc}
\usepackage[T1]{fontenc}
\usepackage[british]{babel}
\usepackage{graphicx, xcolor}
\usepackage{booktabs}
\usepackage{enumerate}
\usepackage{amsmath,amsthm,amsfonts,amssymb}
\usepackage[lining]{FiraSans}
\usepackage{libertine}
\usepackage[libertine]{newtxmath}
\setkomafont{disposition}{\firamedium}

\usepackage{tikz}
\usetikzlibrary{topaths,calc,shapes,decorations.pathmorphing, patterns,backgrounds}
\usepackage[caption=false]{subfig}

\usepackage[colorlinks=true, linkcolor=blue, filecolor=magenta, urlcolor=cyan, citecolor=gray, unicode]{hyperref}
\usepackage{cleveref}
\crefname{chapter}{Chapter}{Chapters}
\crefname{section}{Section}{Sections}
\crefname{subsection}{Section}{Sections}
\crefname{subsubsection}{Section}{Sections}
\crefname{figure}{Figure}{Figures}
\crefname{table}{Table}{Tables}
\theoremstyle{definition}
\newtheorem{innerstep}{Step}
\newenvironment{step}[1]
{\renewcommand\theinnerstep{#1}\innerstep}
{\endinnerstep}
\crefname{step}{Step}{Steps}

\crefname{question}{Question}{Questions}
\newtheorem{definition}{Definition}[section]
\crefname{def}{Definition}{Definitions}

\crefname{ex}{Example}{Examples}

\theoremstyle{plain}
\newtheorem{thm}[definition]{Theorem}
\crefname{thm}{Theorem}{Theorems}
\newtheorem{conj}[definition]{Conjecture}
\crefname{conj}{Conjecture}{Conjectures}
\newtheorem{lemma}[definition]{Lemma}
\crefname{lemma}{Lemma}{Lemmas}
\newtheorem{cor}[definition]{Corollary}
\crefname{cor}{Corollary}{Corollaries}

\crefname{prop}{Proposition}{Propositions}
\newtheorem{obs}[definition]{Observation}
\crefname{obs}{Observation}{Observations}

\crefname{obs}{Observation}{Observations}

\crefname{claim}{Claim}{Claims}

\crefname{claimnr}{Claim}{Claims}
\newtheorem{problem}[definition]{Problem}
\crefname{problem}{Problem}{Problems}

\theoremstyle{remark}

\crefname{rmk}{Remark}{Remarks}
\newcommand{\N}{\mathbb{N}}

\newcommand{\e}{\varepsilon}

\newcommand{\cC}{\mathcal{C}}

\newcommand{\cF}{\mathcal{F}}
\newcommand{\cG}{\mathcal{G}}
\newcommand{\cH}{\mathcal{H}}

\newcommand{\cP}{\mathcal{P}}

\newcommand{\bP}{\mathbf{P}}

\newcommand{\Lk}{\:\mathrm{Lk}}
\newcommand{\den}[1]{\mathrm{d}\left(#1\right)}
\DeclareMathOperator{\dens}{d}
\newcommand{\card}[1]{\left| #1 \right|}
\newcommand{\tp}[1]{\mathrm{tp}\left( #1 \right)}

\begin{document}
\maketitle
\begin{abstract}
  Confirming a conjecture of Gy\'arf\'as, we prove that, for all natural numbers $k$ and $r$, the vertices of every $r$-edge-coloured complete $k$-uniform hypergraph can be partitioned into a bounded number (independent of the size of the hypergraph) of monochromatic tight cycles.
  We further prove that, for all natural numbers $p$ and $r$, the vertices of every $r$-edge-coloured complete graph can be partitioned into a bounded number of $p$-th powers of cycles, settling a problem of Elekes, Soukup, Soukup and Szentmikl\'ossy.
  In fact we prove a common generalisation of both theorems which further extends these results to all host hypergraphs of bounded independence number.
\end{abstract}
\thispagestyle{empty}

\section{Introduction and main results}
%\subsection{History}

A conjecture of Lehel states that the vertices of any $2$-edge-coloured complete graph on $n$ vertices can be partitioned into two monochromatic cycles of different colours. Here single vertices and edges are considered cycles. This conjecture first appeared in~\cite{lehel}, where it was also proved for some special types of colourings of $K_n$.
 \L{}uczak, R\"odl and Szemer\'edi  \cite{Luczak1998} proved Lehel's conjecture for all sufficiently large $n$  using the regularity method. In \cite{Allen2008}, Allen gave an alternative proof, which gave a~better bound on $n$. Finally, Bessy and Thomass\'e \cite{Bessy2010} proved Lehel's conjecture for all integers $n\ge 1$.

For colourings with more than two colours (all colourings in this paper are edge-colourings), Erd\H os, Gy\'arf\'as and Pyber \cite{Erdos1991} proved that the vertices of every $r$-coloured complete graph can be partitioned into $O(r^2 \log r)$ monochromatic cycles and conjectured that $r$ cycles should always suffice.
Their conjecture was refuted by Pokrovskiy~\cite{Pokrovskiy2014}, who showed that, for every $r \geq 3$, there exist infinitely many $r$-coloured complete graphs which cannot be vertex-partitioned into $r$ monochromatic cycles. Pokrovskiy also proposed the following alternative version of Erd\H os, Gy\'arf\'as and Pyber conjecture, which is still widely open.

\begin{conj}[Pokrovskiy \cite{Pokrovskiy2014}]\label{conj:Pokrovskiy} In every $r$-edge-coloured complete graph, there are $r$ vertex-disjoint monochromatic cycles covering all but $c_r$ vertices, where $c_r$ is a constant depending only on $r$.
\end{conj}
\noindent The best-known result for general $r$ is due to Gy\'arf\'as, Ruszink\'o, S\'ark\"ozy and Szemer\'edi \cite{Gyarfas2006}, who showed that the vertices of every large enough $r$-coloured complete graph can be partitioned into at most $100 r \log r$ monochromatic cycles.

Similar partitioning problems have been considered for other graphs, for example,  powers of cycles. Given a graph $H$ and a natural number $p$, the $p$-th power of $H$ is the graph obtained from $H$ by putting an edge between any two vertices whose distance is at most $p$ in $H$. Grinshpun and S\'ark\"ozy~\cite{Grinshpun2016} proved that the vertices of every two-coloured complete graph can be partitioned into at most $2^{cp\log p}$ monochromatic $p$-th powers of cycles, where $c$ is an absolute constant. %In fact, they proved a much more general result about sequences of graphs of bounded degree.
They conjectured that a much smaller number of pieces should suffice, which was confirmed by S\'ark\"ozy~\cite{Sarkoezy2017}. %A not yet published result of Jozef et al.\ shows that xxx powers of cycles suffice.
For more than two colours not much is known. Elekes, D. Soukup, L. Soukup and Szentmikl\'ossy \cite{Elekes2017} proved an analogue of the result of Grinshpun and S\'ark\"ozy for infinite graphs and multiple colours and asked whether it is true for finite graphs.

\begin{problem}[Elekes et al. {\cite[Problem 6.4]{Elekes2017}}\footnote{The problem is phrased differently in \cite{Elekes2017} but this version is stronger, as Elekes et al.\ explain below the problem.}]
  Prove that for every $r,p \in \N$, there is some $c = c(r,p)$ such that the vertices of every $r$-edge-coloured complete graph can be partitioned into at most $c$ monochromatic $p$-th powers of cycles.
\end{problem}
\noindent We shall prove a substantial generalisation of this problem, see \cref{cor:main}.

Another possible generalisation is to study questions about monochromatic partitions for hypergraphs. A $k$-uniform hypergraph ($k$-graph) consists of a vertex set $V$ and a set of $k$-element subsets of $V$. The loose $k$-uniform cycle of length $m$ is the $k$-graph consisting of $m(k-1)$ cyclically ordered vertices and $m$ edges, each edge formed of $k$ consecutive vertices, so that consecutive edges intersect in exactly one vertex. The tight $k$-uniform cycle of length $m$ is the $k$-graph with $m$ cyclically ordered vertices in which any $k$ consecutive vertices form an edge. Loose and tight paths are defined in a similar way. For technical reasons we consider single vertices both as tight and loose cycles and paths.

Questions about monochromatic partitions for hypergraphs were first studied by Gy\'arf\'as and S\'ark\"ozy \cite{Gyarfas2013} who showed that for every $k$, $r \in \N$, there is some $c = c(k,r)$ so that the vertices of every $r$-edge-coloured complete $k$-graph can be partitioned into at most $c$ loose cycles. Later, S\'ark\"ozy \cite{sarkoezy2014} showed that $c(k,r)$ can be be chosen to be $50rk \log (rk)$. Gy\'arf\'as conjectured that a similar result can be obtained for tight cycles.

\begin{conj}[Gy\'arf\'as \cite{Gyarfas2016}]
For every $k,r \in \N$, there is some $c = c(k,r)$ so that the vertices of every $r$-edge-coloured complete $k$-graph can be partitioned into at most $c$ monochromatic tight cycles.
\end{conj}

We shall prove this conjecture and a generalisation in which we allow the host- hypergraph to be any $k$-graph with bounded independence number (i.e.\ without a large set of vertices containing no edges).

\begin{thm}\label{thm:main}
  For every $k,r,\alpha \in \N$, there is some $c = c(k,r,\alpha)$ such that the vertices of every $r$-edge-coloured $k$-graph $G$ with independence number $\alpha(G) \leq \alpha$ can be partitioned into at most $c$ monochromatic tight cycles.
\end{thm}
\noindent We note that a similar result for graphs was obtained by S\'ark\"ozy \cite{Sarkoezy2011}, and for loose cycles in hypergraphs by Gy\'arf\'as and S\'ark\"ozy \cite{Gyarfas2014}.

As a corollary we obtain the following extension of Theorem \ref{thm:main} to $p$-th powers of tight cycles. Here the $p$-th power of a $k$-uniform tight cycle is the $k$-graph obtained by replacing every edge of the $(k+p-1)$-uniform tight cycle by the complete $k$-graph on $k+p-1$ vertices.
\begin{cor}\label{cor:main}
  For every $k,r,p,\alpha \in \N$, there is some $c = c(k,r,p,\alpha)$ such that the vertices of every $r$-edge-coloured $k$-graph $G$ with $\alpha(G) \leq \alpha$ can be partitioned into at most $c$ monochromatic $p$-th powers of tight cycles.
\end{cor}
Since Corollary \ref{cor:main} follows from Theorem \ref{thm:main} easily, we present its short proof here.
\begin{proof}[Proof of Corollary \ref{cor:main}]
For positive integers $k,r,s_1,\ldots,s_r$, let $R_r^{(k)}(s_1, \ldots, s_r)$ denote the $r$-colour Ramsey number for $k$-graphs, that is the smallest positive integer $n$, so that in every $r$-colouring of the complete $k$-graph on $n$ vertices, there is some $i \in [r]$ and $s_i$ distinct vertices which induce a monochromatic clique in colour $i$.

  Let $f(k,r,\alpha)$ be the smallest $c$ for which \cref{thm:main} is true and let $g(k,r,p,\alpha)$ be the smallest $c$ for which \cref{cor:main} is true. We will show that $g(k,r,p,\alpha) \leq f(k+p-1,r,\tilde \alpha)$, where $\tilde \alpha = R_{r+1}^{(k)}\left(k+p-1, \ldots, k+p-1, \alpha+1 \right)-1$.
  Suppose now we are given an $r$-edge-coloured $k$-graph $G$ with $\alpha(G) \leq \alpha$. Define a $(k+p-1)$-graph $H$ on the same vertex-set whose edges are the monochromatic cliques of size $k+p-1$ in $G$. By construction we have $\alpha(H) \leq \tilde \alpha$ and thus, by \cref{thm:main}, there are at most $f(k+p-1,r,\tilde \alpha)$ monochromatic tight cycles partitioning $V(H)$.
  To conclude, note that a tight cycle in $H$ corresponds to a $p$-th power of a tight cycle in $G$.
\end{proof}
In the next section, we shall prove Theorem \ref{thm:main}.

%\subsection{Proof outline}

%%%%%%%%%%%%%%%%%%%%%%%%%%%%%%%%%%%%%%%%%%%%%%%%%%%%%%%%%%%%%%%%%%%%%%%%%%%%%%%%%
%%%%%%%%%%%%%%%%%%%%%%%%%%%%%%%%%%%%%%%%%%%%%%%%%%%%%%%%%%%%%%%%%%%%%%%%%%%%%%%%%

\section{The proof of Theorem \ref{thm:main}}

The proof of Theorem \ref{thm:main} combines the absorption method introduced in \cite{Erdos1991} and the regularity method. For the complete host $k$-graph $G$, the proof of Theorem \ref{thm:main} can be summarised as follows.

First, we find a monochromatic $k$-graph $H_0\subset G$ with the following special property: There is some $B \subset V(H_0)$, so that for every $B' \subset B$ there is a tight cycle in $H_0$ with vertices $V(H_0) \setminus B'$. This is explained in \cref{sec:absmethod}.
We then greedily remove vertex-disjoint monochromatic tight cycles from $V(G)\setminus V(H_0)$ until the set of leftover vertices $R$ is very small in comparison to $B$.
Finally, in \cref{sec:abslemma}, we show that the leftover vertices can be absorbed by $H_0$. More precisely, we show that there are constantly many vertex-disjoint tight cycles with vertices in $R \cup B$ which cover all of $R$. This is the crucial part of the paper and also the place where we use tools from the hypergraph regularity method (introduced in \cref{sec:shortpaths}).

In order to prove the main theorem for host $k$-graphs with bounded independence number, we need to iterate the above process a few times. Here the main difficulty is to show that the iteration process stops after constantly many steps. This will be shown in \cref{sec:proof}.
We start with some basic notation about hypergraphs.
%%%%%%%%%%%%%%%%%%%%%%%%%%%%%%%%%%%%%%%%%%%%%%%%%%%%%%%%%%%%%%%%%
\subsection{Notation}
For a set of vertices $V$ and a natural number $k \geq 2$, let $\binom V k$ denote the set of all $k$-element subsets of $V$. Given a subset $E \subset \binom V k$, $H=(V,E)$ is called a $k$-uniform hypergraph ($k$-graph). We sometimes use the notation $H = (V(H), E(H))$.
The \emph{density} of a $k$-graph $H$ with $n$ vertices is given by $\dens(H) = |E(H)|/\binom{n}{k}$.

Let $H$ be a $k$-graph. Given some $e \subset V(H)$ with $1 \leq |e| \leq k$, we define its \emph{degree} of $e$ by $\deg(e) := \card{f \in E(H): e \subset f}$. If $|e|= 1$ for some $v \in V(H)$ we simply write $\deg(v)$ for $\deg(\{v\})$ and if $|e|=k-1$, we call $\deg(e)$ \emph{co-degree}.
Given a partition $\cP=\left\{ V_1, \ldots, V_t \right\}$ of $V$, we say that $H$ is $\cP$-partite if $| e \cap V_i | \leq 1$ for every $e \in E(H)$ and every $i \in [t]$. The $k$-graph $H$ is $s$-partite if it is $\cP$-partite for some partition $\cP$ of $V$ with $s$ parts.
We denote by $K^{(k)}(\cP)$ the complete $\cP$-partite $k$-graph.
Furthermore, given some $2 \leq j \leq k-1$ and a $j$-graph $H$, we define $K^{(k)}(H)$ to be the set of all $k$-cliques in $H^{(j)}$, seen as a $k$-graph on $V$.

Given a $k$-graph $H$ and $\ell \leq k$ distinct vertices $v_1, \ldots, v_\ell \in V(H)$, we define the \emph{link-graph} $\Lk_{H}(v_1, \ldots, v_\ell)$ as the $(k-\ell)$-graph on $V(H) \setminus \{v_1, \ldots, v_\ell\}$ with edges $\{e \in \binom{V(H)}{k-\ell}: e \cup \{v_1, \ldots, v_l\} \in E(H)\}$. If, in addition, disjoint sets $V_1, \ldots, V_{k-\ell} \subset V(H) \setminus \{v_1, \ldots, v_\ell\}$ are given, we denote by $\Lk_{H}(v_1, \ldots, v_\ell ; V_1, \ldots, V_{k-\ell})$ the $(k-\ell)$-partite $(k-\ell)$-graph with parts $V_1, \ldots, V_{k-\ell}$ and edges $\{e \in K^{(k-\ell)}(V_1, \ldots, V_{k-\ell}): e \cup \{v_1, \ldots, v_\ell\} \in E(H)\}$. If there is no danger of confusion, we drop the subscript $H$.

%%%%%%%%%%%%%%%%%%%%%%%%%%%%%%%%%%%%%%%%%%%%%%%%%%%%%%%%%%%%%%%%%%%%%%%%%%%%%%%%%

\subsection{Finding short paths}\label{sec:shortpaths}
The goal of this section is to prove the following lemma, which allows us to find in any dense $k$-graph $G$, a dense subgraph $H \subset G$ in which any two non-isolated $(k-1)$-sets are connected by a short path of a given prescribed length. For this, we need to use basic tools from hypergraph regularity, but the reader may use \cref{lem:shortpaths} as a black box if she would like to avoid it.

Before stating the lemma, we need to introduce some notation. Fix some $k\geq 2$ and a partition $\cP=\{V_1, \ldots, V_k\}$. We call a tight path in $K^{(k)}(\cP)$ \emph{positively oriented} if its vertex sequence $(u_1, \ldots, u_m)$ travels through $\cP$ in cyclic order, i.e.\ there is some $j \in [k]$ such that $u_i \in V_{i+j}$ for every $i \in [m]$, where we identify $k+1 \equiv 1$. In this subsection, we will only consider positively oriented tight cycles. In particular, given some $e  \in K^{(k-1)}(\cP)$, the ordering of $e$ in a tight path starting at $e$ is uniquely determined.

\begin{lemma}\label{lem:shortpaths}
  For every $d > 0$, there are constants $\delta=\delta(d)>0$ and $\gamma=\gamma(d)>0$, such that the following is true for every partition $\cP = \left\{ V_1, \ldots, V_k \right\}$ and every $\cP$-partite $k$-graph $G$ of density at least $d$. There is a $\cP$-partite sub-$k$-graph $H \subset G$ of density at least $\delta$ such that for  every set $S=S_1 \cup \ldots \cup S_k$ with $S_i \subset V_i$ and $\card{S_i}\leq \gamma \card{V_i}$ and any two $e, f \in K^{(k-1)}(\cP)$ which are disjoint from $S$ and have positive co-degree, there is a positively oriented tight path of length $\ell \in \left\{ k+2, \ldots, 2k+1 \right\}$ in $H$ which starts at $e$, ends at $f$ and avoids $S$.%\footnote{More precisely, $ \ell = k + \tp{e,f}$ if $\tp{e,f} \geq 2$ and $ \ell = 2k + \tp{e,f}$ otherwise.}
  %and a $\cP$-partite $(k-1)$-graph $H^{(k-1)}$ with the following properties.
  %\begin{enumerate}
  %  \item $\deg(e)=0$ for every $e \in K^{(k-1)}(\cP) \setminus H^{(k-1)}$.
  %  \item For every set $S=S_1 \cup \ldots \cup S_k$ with $S_i \subset V_i$ and $\card{S_i}\leq \sigma \card{V_i}$ and any two $e, f \in H^{(k-1)}$ which are disjoint from $S$, there is a tight path of length $\ell \in \left\{ k+2, \ldots, 2k+1 \right\}$ in $H^{(k)}$ which starts at $e$, ends at $f$ and avoids $S$.
  %\end{enumerate}
\end{lemma}

Note that the length of the cycle in \cref{lem:shortpaths} is uniquely determined by the types of $e$ and $f$. The \emph{type} of $e \in K^{(k-1)}(\cP)$, denoted by $\tp{e}$, is the unique index $i \in [k]$ such that $e \cap V_i = \emptyset$.
Given two $(k-1)$-sets $e,f \in K^{(k-1)}(\cP)$, the type of $(e,f)$ is given by $\tp{e,f}:=\tp f - \tp e  (\mathrm{mod} \ k)$. It is easy to see that every tight path in $K^{(k)}(\cP)$ which starts at $e$ and ends at $f$ has length $\ell k + \tp{e,f}$ for some $\ell \geq 0$. In particular, in \cref{lem:shortpaths}, we have $ \ell = k + \tp{e,f}$ if $\tp{e,f} \geq 2$ and $ \ell = 2k + \tp{e,f}$ otherwise.

\subsubsection{Hypergraph regularity}
We will now introduce the basic concepts of hypergraph regularity in order to state a simple consequence of the strong hypergraph regularity lemma which guarantees a dense regular complex in every large enough $k$-graph.

For technical reasons, we want to see a $1$-graph on some vertex-set $V$ as a partition of $V$ in what follows. We call $\cH^{(k)}=(H^{(1)},\ldots,H^{(k)})$ a $k$-complex if $H^{(j)}$ is a $j$-graph for every $j \in [k]$ and $H^{(j)}$ underlies $H^{(j+1)}$, i.e.\ $H^{(j+1)} \subset K^{(j+1)}(H^{(j)})$ for every $j \in [k-1]$. Note that, in particular, $H^{(j)}$ is $H^{(1)}$-partite for every $j \in [k]$. We call $\cH^{(k)}$ $s$-partite if $H^{(1)}$ consists of $s$ parts.

Now, given some $j$-graph $H^{(j)}$ and some underlying $(j-1)$-graph $H^{(j-1)}$, we define the \emph{density} of $H^{(j)}$ w.r.t.\ $H^{(j-1)}$ by
\[ \den{H^{(j)}|H^{(j-1)}} = \frac{\card{H^{(j)} \cap K^{(j)}(H^{(j-1)})}}{{\card{K^{(j)}(H^{(j-1)})}}}.\]
We are now ready to define regularity.

\begin{definition}\
  \begin{itemize}
    \item Let $r,j \in \N$ with $j \geq 2$, $\e,d_j >0$, and $H^{(j)}$ be a $j$-partite $j$-graph and $H^{(j-1)}$ be an underlying ($j$-partite) $(j-1)$-graph. We call $H^{(j)}$ $(\e,d_j,r)$-regular w.r.t.\ $H^{(j-1)}$ if for all $Q_1^{(j-1)}, \ldots, Q_r^{(j-1)} \subset E(H^{(j-1)}) $, we have
      \[ \card{\bigcup\nolimits_{i \in [r]} K^{(j)}\left(Q_i^{(j-1)}\right)} \geq \e \card{K^{(j)}\left(H^{(j-1)}\right)} \implies \card{\den{H^{(j)} \middle| \bigcup\nolimits_{i \in [r]} Q_i^{(j-1)}}-d_j} \leq \e.\]
     We simply say $(\e,*,r)$-regular for $(\e, \den{H^{(j)}|H^{(j-1)}}, r)$-regular and $(\e,d)$-regular for $(\e,d,1)$-regular.
    \item Let $j,s \in \N$ with $s \geq j \geq 2$, $\e,d_j >0$, and $H^{(j)}$ be an $s$-partite $j$-graph and $H^{(j-1)}$ be an underlying ($s$-partite) $(j-1)$-graph. We call $H^{(j)}$ $(\e,d_j)$-regular w.r.t.\ $H^{(j-1)}$ if $H^{(j)}[V_1,\ldots,V_j]$ is $(\e,d_j)$-regular w.r.t.\ $H^{(j-1)}[V_{i_1}, \ldots, V_{i_j}]$ for all $1 \leq i_1 < \ldots < i_j \leq s$, where $\{V_1, \ldots, V_s\}$ is the vertex partition of $V(H^{(j)})$.
    \item Let $k,r \in \N$, $\e,\e_k,d_2,\ldots,d_k >0$, and $\cH^{(k)}=(H_1, \ldots, H_k)$ be a $k$-partite $k$-complex. We call $\cH^{(k)}$ $(d_2,\ldots,d_k,\e,\e_k,r)$-regular, if $H^{(j)}$ is $(\e,d_j)$-regular with respect to $H^{(j-1)}$ for every $j = 2,\ldots,k-1$ and $H^{(k)}$ is $(\e_k,d_k,r)$-regular w.r.t.\ $H^{(k-1)}$.
  \end{itemize}
\end{definition}

The following theorem is a direct consequence of the strong hypergraph regularity lemma as stated in \cite{Roedl2007} (with the exception that we allow for an initial partition of not necessarily equal sizes).

\begin{thm}
\label{thm:reg-complex}
For all integers $k \geq 2$, constants $\e_k > 0$, and functions $\e: (0,1) \to (0,1)$ and $r: (0,1) \to \N$, there exists some $\delta=\delta(k,\e,\e_k,r) > 0$ such that the following is true.
For every partition $\cP = \{V_1, \ldots, V_k \}$ of some set $V$ and every $\cP$-partite $k$-graph $G^{(k)}$ of density $d \geq 2\e_k$, there are sets $ U_i \subset V_i$ with $ \card{U_i} \geq \delta \card{V_i}$ for every $i \in [k]$
and constants $d_2, \ldots, d_{k-1} \geq \delta$ and $d_k \geq d/2$ for which there exists some $(d_2,\ldots,d_k,\e(\delta),\e_k,r(\delta))$-regular $k$-complex $\cH^{(k)}$, so that $H^{(1)}=\{U_1, \ldots, U_k\}$.
\end{thm}

We will use the following special case of the extension lemma in \cite[Lemma 5]{Cooley2009} to find short tight paths between almost any two $(k-1)$-sets in a regular complex. Fix a $(d_2,\ldots,d_k,\e,\e_k)$-regular complex $H^{(k)}=(\cP,H^{(2)}, \ldots, H^{(k)})$, where $\cP=\left\{ V_1, \ldots, V_k \right\}$. Let $H_i^{(k-1)} \subset H^{(k-1)}$ denote the edges of type $i$ and note that the dense counting lemma for complexes \cite[Lemma 6]{Cooley2009} implies that
\[ \card{H^{(k-1)}_{i_0}} = (1 \pm \e) \prod_{j=2}^{k-1} d_j^{\binom{k-1}{j}} \prod_{i \in [k] \setminus i_0} \card{V_i}.\]

Given some $\beta >0$, we call a pair $\left(e,f\right) \in H_{i_1}^{(k-1)} \times H^{(k-1)}_{i_2}$ $\beta$-typical for $\cH^{(k)}$ if the number of tight paths of length $\ell := k + \tp{i_1,i_2}$ in $H^{(k)}$ which start at $e$ and end at $f$ is
\[(1 \pm \beta) \prod_{j=2}^k d_j^{\ell \binom{k-1}{j-1} - \binom{k-1}{j}} \prod_{i \in \{i_1,\ldots, i_2\}} \card{V_i},\]
where $\left\{ i_1, \ldots, i_2 \right\}$ is understood in cyclic ordering.
Note here that the number of $j$-sets in a $k$-uniform tight path of length $\ell$ which are contained in an edge is $\ell \binom{k-1}{j-1} + \binom{k-1}{j}$. However, $2 \binom{k-1}{j}$ of these are contained in $e$ (the first $k-1$ vertices) or $f$ (the last $k-1$ vertices), which are already fixed in our example.

\begin{lemma}
  Let $k,r,n_0 \in \N$, $\beta,d_2,\ldots,d_k,\e,\e_k > 0$ and suppose that
  \[ 1/n_0 \ll 1/r, \e \ll \min\{\e_k,d_2,\ldots,d_{k-1}\} \leq \e_k \ll \beta,d_k,1/k.\]
  Then the following is true for all integers $n \geq n_0$, for all indices $i_1,i_2 \in [k]$ and every $(d_2,\ldots,d_k,\e,\e_k,r)$-regular complex $\cH^{(k)}=\left(H^{(1)},\ldots,H^{(k)}\right)$ with $\card{V_i} \geq n_0$ for all $i \in [k]$, where $H^{(1)}=\left\{ V_1, \ldots, V_k \right\}$. All but at most $\beta \card{H_{i_1}^{(k-1)}}  \card{H_{i_2}^{(k-1)}}$ pairs $\left(e,f\right) \in H_{i_1}^{(k-1)} \times H^{(k-1)}_{i_2}$ are $\beta$-typical for $\cH^{(k)}$.
  \label{lem:extension}
\end{lemma}

Combining \cref{thm:reg-complex,lem:extension} gives \cref{lem:shortpaths}.

\begin{proof}[Proof-sketch of \cref{lem:shortpaths}.]
  Apply \cref{thm:reg-complex} with suitable constants and delete all $e \in H^{(k-1)}$ of small co-degree. Let $e \in H_{i_1}^{(k-1)}$ and $f \in  H_{i_2}^{(k-1)}$ for some $i_1,i_2 \in [k]$ and define
\begin{align*}
  X &= \left\{ g^{(k-1)} \in H_{i_1+1}^{(k-1)}:e \cup g^{(k-1)} \in H^{(k)}  \right\} \text{ and} \\
  Y &= \left\{ g^{(k-1)} \in H_{i_2-1}^{(k-1)}:f \cup g^{(k-1)} \in H^{(k)}  \right\}.
\end{align*}
   Let $\tilde X \subset X$ and $\tilde Y \subset Y$ be the sets of all those edges in $X$ and $Y$ avoiding $S$. By \cref{lem:extension} at least one pair in $\tilde X \times \tilde Y$ must be typical and by a counting argument not all of the promised paths can touch $S$.
\end{proof}
%%%%%%%%%%%%%%%%%%%%%%%%%%%%%%%%%%%%%%%%%%%%%%%%%%%%%%%%%%%%%%%%%%
\subsection{Absorption Method}\label{sec:absmethod}
The idea of the absorption method is to first cover almost every vertex by vertex-disjoint monochromatic tight cycles and then absorb the leftover using a suitable absorption lemma.

\begin{lemma}\label{lem:greedy}
  For all $k,r,\alpha \in \N$ and every $ \gamma > 0 $, there is some $ c = c(k,r,\alpha, \gamma)$ so that the following is true for every $r$-coloured $k$-graph $G$ on $n$ vertices with $\alpha(G) \leq \alpha$. There is a collection of at most $c$ vertex-disjoint monochromatic tight cycles whose vertices cover all but at most $ \gamma n $ vertices.
\end{lemma}

\begin{definition}
Let $G$ be a hypergraph, $ \chi $ be a colouring of $E(G)$ and $A,B \subset V(G)$ disjoint subsets. $A$ is called an \emph{absorber} for $B$ if there is a monochromatic tight cycle with vertices $A \cup B'$ for every $B' \subset B$.
\end{definition}

\begin{lemma}\label{lem:crownk}
  For every $k,r,\alpha\in \N$, there is some $\beta = \beta(k,r,\alpha) > 0$ such that the following is true for every $k$-graph $G$ with $\alpha(G) \leq \alpha$. In every $r$-colouring of $E(G)$ there are disjoint sets $A,B \subset V(G)$ with $|B| \geq \beta |V(G)|$ such that $A$ absorbs $B$.
\end{lemma}

The following hypergraph will function as our absorber. A very similar hypergraph was used by Gy\'arf\'as and S\'ark\"ozy to absorb loose cycles \cite{Gyarfas2013,Gyarfas2014}. See \cref{fig:crown} for an example.

\begin{definition}\label{lem:crown}
  The ($k$-uniform) crown of order $t$, $T_{t}^{(k)}$, is a tight cycle with $ n=t(k-1)$ vertices $ v_0, \ldots, v_{n-1}$ (the base) and additional vertices $ u_0, \ldots, u_{t-1}$ (the rim). Furthermore, for each $ i=0, \ldots, t-1$, we add the $ k$ edges $ \{u_i,v_{(k-1)i+j}, \ldots,v_{(k-1)i+j +k -2} \}$, $ j=0, \ldots, k-1$.
\end{definition}

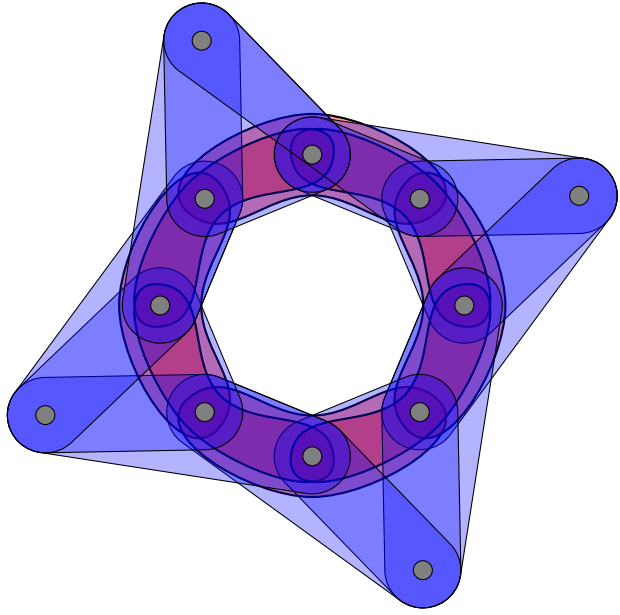
\begin{figure}[ht]
\centering

\begin{tikzpicture}
\def \r {2cm}
\def \n {8}
\def \vxsize {0.25cm}

\newcommand{\convexpath}[2]{
	[   
	create hullcoords/.code={
		\global\edef\namelist{#1}
		\foreach [count=\counter] \nodename in \namelist {
			\global\edef\numberofnodes{\counter}
			\coordinate (hullcoord\counter) at (\nodename);
		}
		\coordinate (hullcoord0) at (hullcoord\numberofnodes);
		\pgfmathtruncatemacro\lastnumber{\numberofnodes+1}
		\coordinate (hullcoord\lastnumber) at (hullcoord1);
	},
	create hullcoords
	]
	($(hullcoord1)!#2!-90:(hullcoord0)$)
	\foreach [
	evaluate=\currentnode as \previousnode using \currentnode-1,
	evaluate=\currentnode as \nextnode using \currentnode+1
	] \currentnode in {1,...,\numberofnodes} {
		let \p1 = ($(hullcoord\currentnode) - (hullcoord\previousnode)$),
		\n1 = {atan2(\y1,\x1) + 90},
		\p2 = ($(hullcoord\nextnode) - (hullcoord\currentnode)$),
		\n2 = {atan2(\y2,\x2) + 90},
		\n{delta} = {Mod(\n2-\n1,360) - 360}
		in 
		{arc [start angle=\n1, delta angle=\n{delta}, radius=#2]}
		-- ($(hullcoord\nextnode)!#2!-90:(hullcoord\currentnode)$) 
	}
}

\pgfdeclarelayer{bg}    % declare background layer
\pgfsetlayers{bg,main}  % set the order of the layers (main is the standard layer)

\foreach \i in {1,...,\n} {
	\pgfmathsetmacro\angle{(\i -1)*360/\n}
	\node[draw=black,fill=black!50, circle, minimum width = \vxsize, inner sep=0cm] (v\i) at (\angle:\r) {};
}

%\hyperedge{v1}{v2}{v3}

\foreach \i in {2,4,6,8} {
	\pgfmathsetmacro\angle{(\i -1.5)*360/\n}
	\node[draw=black,fill=black!50, circle, minimum width = \vxsize, inner sep=0cm] (u\i) at (\angle:1.9*\r) {};
}

\begin{pgfonlayer}{bg}
\begin{scope}[fill opacity=0.4]
\foreach \i in {2,4,...,\n} {
	\pgfmathsetmacro\anglea{(\i -1)*360/\n - 8}
	\pgfmathsetmacro\angleb{(\i)*360/\n}
	\pgfmathsetmacro\anglec{(\i +1)*360/\n + 8}
	\filldraw[fill=red, draw=black, thick] (\anglea:\r) 
	to[out = \anglea, in = \angleb + 270] (\angleb:1.27*\r) 
	to[out = 90 + \angleb, in = \anglec] (\anglec:\r)
	to[out = 180 + \anglec, in =90+ \angleb] (\angleb:0.73*\r)
	to[out = 270 + \angleb, in = 180+\anglea] (\anglea:\r);
}
\end{scope}
\end{pgfonlayer}

\begin{pgfonlayer}{bg}
\begin{scope}[fill opacity=0.4]
\foreach \i in {1,3,...,\n} {
	\pgfmathsetmacro\anglea{(\i -1)*360/\n - 8}
	\pgfmathsetmacro\angleb{(\i)*360/\n}
	\pgfmathsetmacro\anglec{(\i +1)*360/\n + 8}
	\draw[fill=red,draw=black, thick] (\anglea:\r) 
	to[out = \anglea, in = \angleb + 270] (\angleb:1.15*\r) 
	to[out = 90 + \angleb, in = \anglec] (\anglec:\r)
	to[out = 180 + \anglec, in =90+ \angleb] (\angleb:0.85*\r)
	to[out = 270 + \angleb, in = 180+\anglea] (\anglea:\r);
}
\end{scope}
\end{pgfonlayer}

\begin{pgfonlayer}{bg}
\begin{scope}[fill opacity=0.3]
\draw[fill=blue] \convexpath{v2,v3,u2}{0.5cm};
\draw[fill=blue] \convexpath{v2,u2,v1}{0.5cm};
\draw[fill=blue] \convexpath{u2,v8,v1}{0.5cm};
\draw[fill=blue] \convexpath{v1,u8,v8}{0.5cm};
\draw[fill=blue] \convexpath{v8,u8,v7}{0.5cm};
\draw[fill=blue] \convexpath{u8,v6,v7}{0.5cm};
\draw[fill=blue] \convexpath{v7,u6,v6}{0.5cm};
\draw[fill=blue] \convexpath{v6,u6,v5}{0.5cm};
\draw[fill=blue] \convexpath{u6,v4,v5}{0.5cm};
\draw[fill=blue] \convexpath{v5,u4,v4}{0.5cm};
\draw[fill=blue] \convexpath{v4,u4,v3}{0.5cm};
\draw[fill=blue] \convexpath{u4,v3,v2}{0.5cm};
\end{scope}
\end{pgfonlayer}
\end{tikzpicture}

\caption{A $3$-uniform crown of order $4$. The edges of the tight cycle are red and the remaining edges are blue.}
\label{fig:crown}
\end{figure}

It is easy to see that the base of a crown is an absorber for the rim. To prove \cref{lem:crownk}, we therefore only need to show that we can always find monochromatic crowns of linear size. Both this and \cref{lem:greedy} are consequences of the following theorem of Cooley, Fountoulakis, K\"uhn, and Osthus \cite{Cooley2009} (see also \cite{Ishigami2007} and \cite{Conlon2009}).

\begin{thm}\label{thm:lin-ram}
  For every $r,k,\Delta \in \N$, there is some $C=C(r,k,\Delta)>0$ so that the following is true for all k-graphs $H_1, \ldots, H_r$ with at most $n$ vertices and maximum degree at most $\Delta$, and every $N \geq Cn$. In every edge-colouring of $K_N^{(k)}$ with colours $c_1, \ldots, c_r$, there is some $i \in [r]$ for which there is a $c_i$-monochromatic copy of $H_i$.
\end{thm}

\begin{proof}[Proof of \cref{lem:crownk}]
  Suppose $k,r,\alpha$ and $G$ are given as in the theorem and that $E(G)$ is coloured with $r$ colours. Let $N = \card{V(G)}$, $\Delta := \max \left\{ 2k,\binom{\alpha}{k-1} \right\}$ and $c=1/((k-1)C)$ where $C=C(r+1,k,\Delta)$ is given by \cref{thm:lin-ram}. Furthermore, let $n = \card{V(T_{cN}^{(k)})} = N/C$.
  Consider now the $(r+1)$-colouring of $E\left( K_N^{(k)} \right)$ in which every edge in $E(G)$ receives the same colour as in $G$ and every other edge receives colour $r+1$. Let $H_{r+1} = K_{\alpha +1}^{(k)}$ and $ H_i = T_{cN}^{(k)}$ for all $ i \in [t]$, and note that $\Delta(H_i) \leq \Delta$ for all $ i \in [r+1]$. By choice of $\Delta$, there is no monochromatic $H_{r+1}$ in colour $r+1$ and hence, since $N \geq Cn$, there is a monochromatic copy of $H_i$ for some $i \in [r]$. Therefore, there is a monochromatic crown of size $c|V(G)|$ and its base is an absorber for its rim.
\end{proof}

\begin{proof}[Proof of \cref{lem:greedy}]
  Applying \cref{thm:lin-ram} with $r+1$ colours, uniformity $k$, maximum degree $\Delta = \max\{k,\binom{\alpha}{k-1}\}$, and $H_1 = \ldots = H_r$ being tight cycles on $n/(C_{Thm~\ref{thm:lin-ram}}(r+1,k,\Delta))$ vertices and $H_{r+1} = K_{\alpha+1}^{(k)}$ gives the following.
  There exist some $\e = \e(r,k,\alpha)$ so that in every $r$-coloured $k$-graph $G$ on $n$ vertices with $\alpha(G) \leq \alpha$, there is a monochromatic tight cycle on at least $\e n$ vertices.\footnote{Here, we treat non-edges as colour $r+1$ again.} By iterating this process $i$ times, we find $i$ vertex-disjoint monochromatic tight cycles covering all but $(1-\e)^i n$ vertices. This finishes the proof, since $(1-\e)^i \to 0$ as $ i \to \infty$.
\end{proof}

%%%%%%%%%%%%%%%%%%%%%%%%%%%%%%%%%%%%%%%%%%%%%%%%%%%%%%%%%%%%%%%%%%%%%%%%%
\subsection{Absorption Lemma}\label{sec:abslemma}
In this section we prove a suitable absorption lemma for our approach.

\begin{lemma}\label{lem:keyk}
	For every $\e > 0$ and $k,r \in \N$, there is some $\gamma = \gamma(k,r,\e)>0$ and some $c = c(k,r,\e)$ such that the following is true. Let $H$ be a $k$-partite $k$-graph with parts $B_1, \ldots, B_k$ such that $ |B_1| \geq \ldots  \geq |B_{k-1}| \geq |B_k|/\gamma $ and $|\Lk(v;B_1, \ldots, B_{k-1})| \geq \e |B_1| \cdots \cdot |B_{k-1}|$ for every $v \in B_k$. Then, in every $r$-colouring of $E(H)$, there are $c$ vertex-disjoint monochromatic tight cycles covering $B_k$.
\end{lemma}

Note that it is enough to cover all but a bounded number of vertices, since we allow single vertices as tight cycles. We will make use of this in the proof and frequently remove few vertices.

We will use the following theorem of P\' osa \cite{Posa1963}.

\begin{thm}[P\' osa]\label{thm:posa}
  In every graph $G$, there is a collection of at most $\alpha(G)$ cycles whose vertices partition $V(G)$.
\end{thm}

We further need the following simple but quite technical lemma, which states that, given a ground set $X$ and a collection $\cF$ of subsets of $X$ of linear size, we can group almost all of these subsets into groups of size $4$ which have a large common intersection. We will apply this lemma when $X$ is the edge-set of a hypergraph $G$ and $\cF$ is a collection of subgraphs of $G$.

\begin{lemma}\label{lem:blocks}
  For every $\e > 0$ there is some $ \delta = \delta (\e) > 0$ and some $C = C(\e)>0$ such that the following is true for every $m \in \N$. Let $X$ be set of size $m$ and $ \cF \subset 2^X$ be a family of subsets such that $\card F \geq \e m$ for every $F \in \cF$.
  Then there is some $ \cG \subset \cF$ of size $\card{\cG} \leq C$ and a partition $\cP$ of $ \cF \setminus \cG$ into sets of size $4$ such that $\card{\bigcap_{i=1}^4 B_i} \geq \delta m$ for every $\{B_1,B_2,B_3,B_4\} \in \cP$.
\end{lemma}

 We will prove the lemma with $ \delta(\e) = e^4/2^6$ and $C(\e)=8/\e^2+2/\e$.
\begin{proof}
  Define a graph $G$ on $\cF$ by $\{F_1, F_2\} \in E(G)$ if and only if $\card{F_1 \cap F_2} \geq (\e/2)^2 m$. We claim that $\alpha(G) \leq 2/\e$. Suppose for contradiction that there is an independent set $I$ of size $2/\e+1$. Then we have $|F_0 \setminus \bigcup_{F \in I \setminus \{ F_0 \}} F| \geq \e m/2$ for every $F_0 \in I$  and hence $\left| \bigcup_{F \in I} F \right| > m$, a contradiction.

  Since every graph has a matching of size at least $v(G) - \alpha(G)$, we find a matching $\cP_1$ in $G$ of all but at most $2/\e$ vertices of $G$ (i.e.\ $F \in \cF$). Let $\cG_1= \cF \setminus V(\cP_1)$ and note that $\cP_1$ is a partition of $\cF \setminus \cG_1$ into sets of size $2$. Let $\cF_1=\left\{ F_1 \cap F_2 : \left\{ F_1,F_2 \right\}\in \cP_1 \right\}$ and iterate the process once more.
\end{proof}
\begin{figure}[ht]
	\centering
	\subfloat[Link graphs]{%
		\includegraphics[clip, width=.7\textwidth]{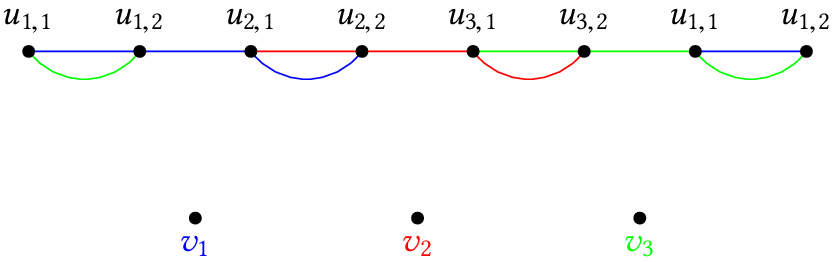}%
	}
	\hfill
	\subfloat[Tight Cycle]{%
		\includegraphics[clip, width=.7\textwidth]{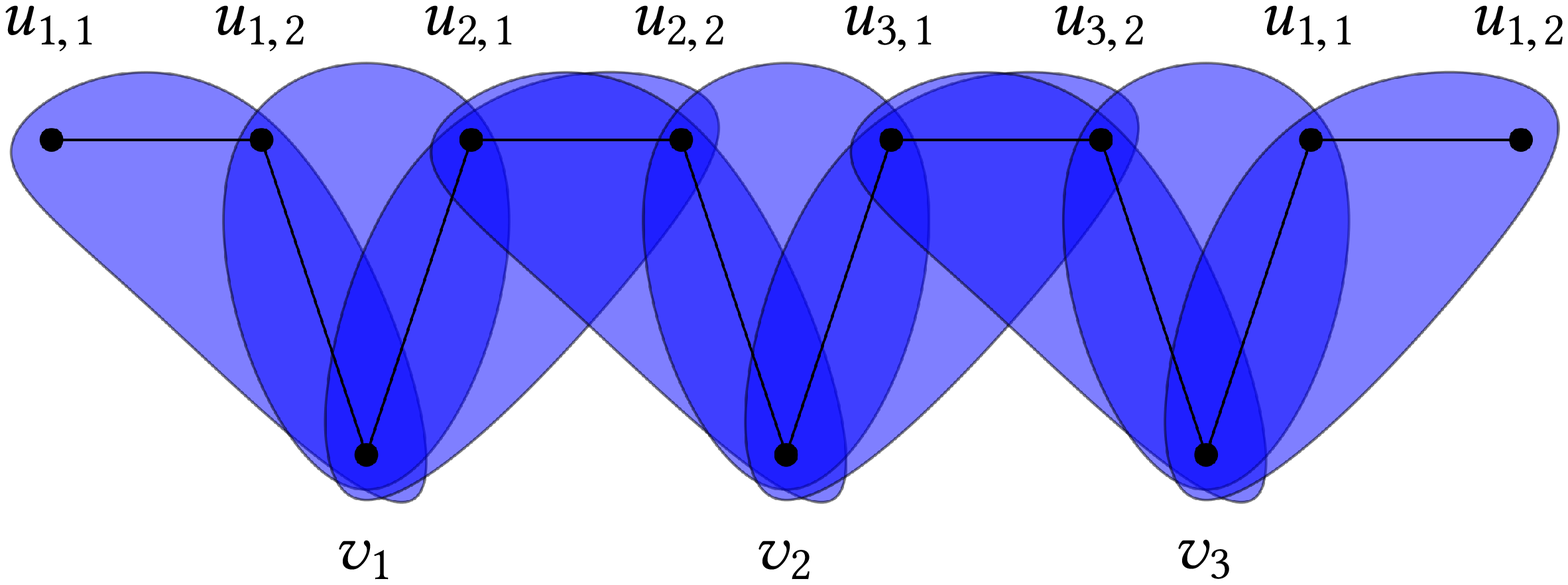}%
	}

	\caption{A sketch of \cref{obs:liftcycle} for $k=t=3$. Figure (a) shows the link graphs of $v_1$ (blue), $v_2$ (red) and $v_3$ (green). The colours are for demonstration purposes only and are not related to the given edge-colouring. Figure (b) shows the resulting tight cycle. In both figures, we identify the ends ($u_{1,1}$ and $u_{1,2}$) to simplify the drawing.}
	\label{fig:liftcycles}
\end{figure}
\begin{proof}[Proof of \cref{lem:keyk}]
By choosing $\gamma$ small enough, we may assume that $|B_1|, \ldots, |B_{k-1}|$ are sufficiently large for the following arguments.
First we claim that it suffices to prove the lemma for $r=1$. Indeed, partition $B_k = B_{k,1} \cup \ldots \cup B_{k,r}$ so that for each $i \in [r]$ and $v \in B_{k,i}$, we have $|\Lk_i(v;B_1, \ldots, B_{k-1})| \geq \e/r \cdot |B_1| \cdots |B_{k-1}|$ and delete all edges containing $v$ whose colour is not $i$. (Here we denote by $\Lk_i(\cdot)$ the link graph with respect to $G_i$, the graph with all edges of colour $i$.)
Next, for each $j \in [k-1]$, partition $B_j = B_{j,1} \cup \ldots \cup B_{j,r}$ into sets of equal sizes so that $|\Lk_i(v;B_{1,i}, \ldots, B_{k-1,i})| \geq \e/(2r) \cdot |B_{1,i}| \cdots |B_{k-1,i}|$. Such a partition can be found for example by choosing one uniformly at random and applying the probabilistic method). Finally, we can apply the one-colour result (with $\e' = \e/(2r)$) for each $i \in [r]$.

Fix $\e>0$, $k \geq 2$ and a $k$-partite $k$-graph $H$ with parts $B_1, \ldots, B_k$ as in the statement of the lemma.
Choose constants $\gamma, \delta_1, \delta_2, \delta_3 >0$ so that $0 < \gamma \ll \delta_3 \ll \delta_2 \ll \delta_1 \ll \e, 1/k $.
We begin with a simple but important observation.
\begin{obs}\label{obs:liftcycle}
Let $v_1, \ldots, v_t \in B_k$ be distinct vertices and $\cC$ be a tight cycle in $K^{(k-1)}\left( B_1, \ldots, B_{k-1} \right)$ with vertex-sequence $(u_{1,1}, \ldots, u_{1,k-1}, \ldots, u_{t,1}, \ldots, u_{t,k-1})$. Denote by $e_{s,i}$ the edge in $C$ starting at $u_{s,i}$ and suppose that
\begin{enumerate}[(i)]
  \item $e_{s,i} \in \Lk\left( v_s; B_1, \ldots, B_{k-1} \right)$ for every $s \in [t]$ and every $i \in [k-1]$ and
  \item $e_{s,1} \in \Lk\left( v_{s-1}; B_1, \ldots, B_{k-1} \right)$ for every $s \in [t]$ (here $v_0 := v_t$).
\end{enumerate}
Then, $(u_{1,1}, \ldots, u_{1,k-1}, v_1, \ldots, u_{t,1}, \ldots, u_{t,k-1}, v_t)$ is the vertex-sequence of a tight cycle in $H$.
\end{obs}

The proof of \cref{obs:liftcycle} follows readily from the definition of the link graphs. See \cref{fig:liftcycles} for an overview. We will now proceed in three steps. For simplicity, we write $H_v := \Lk_{H}(v;B_1, \ldots, B_{k-1})$ for $v \in B_k$.

\begin{step}{1 (Divide into blocks)}
By \cref{lem:blocks}, there is some $C=C(\e) \in \mathbb N$ and a partition $\cP$ of all but $C$ graphs from $\{H_v : v \in B_k\}$ into \emph{blocks} $\cH$ of size $4$ with $e(\cH) := |\bigcap_{H \in \cH} E(H)| \geq \delta_1 |B_1| \cdots |B_{k-1}|$ for every $\cH \in \cP$. Remove the $C$ leftover vertices from $B_k$.

Think of every block $\cH$ now as a $(k-1)$-graph with edges $E(\cH) := \bigcap_{H \in \cH} E(H)$. By applying \cref{lem:shortpaths} (with $k-1$ instead of $k$), for each $\cH \in \cP$, we find a subgraph $\cH' \subset \cH$ such that
$e(\cH') \geq \delta_2 \card{B_1} \cdots \card{B_{k-1}}$ with the same property as in \cref{lem:shortpaths}. By deleting all the edges of $ \cH \setminus \cH'$ we may assume that $\cH$ itself has this property.
\end{step}

\begin{step}{2 (Cover blocks by paths)}
Define an auxiliary graph $G$ with $V(G) = \cP$ and $\{\cH_{1},\cH_{2}\} \in E(G)$ if and only if $e(\cH_{1} \cap \cH_{2}) \geq \delta_3 |B_1|\cdots |B_{k-1}|$. Similarly as in the proof of \cref{lem:blocks}, we conclude that $\alpha(G) \leq 2/\delta_2$, and hence $V(G)$ can be covered by $2 / \delta_2$ vertex-disjoint paths using \cref{thm:posa}.
\end{step}

\begin{figure}[ht]
	\centering
	\subfloat[A path of blocks.]{%
		\includegraphics[clip,width=.95\textwidth]{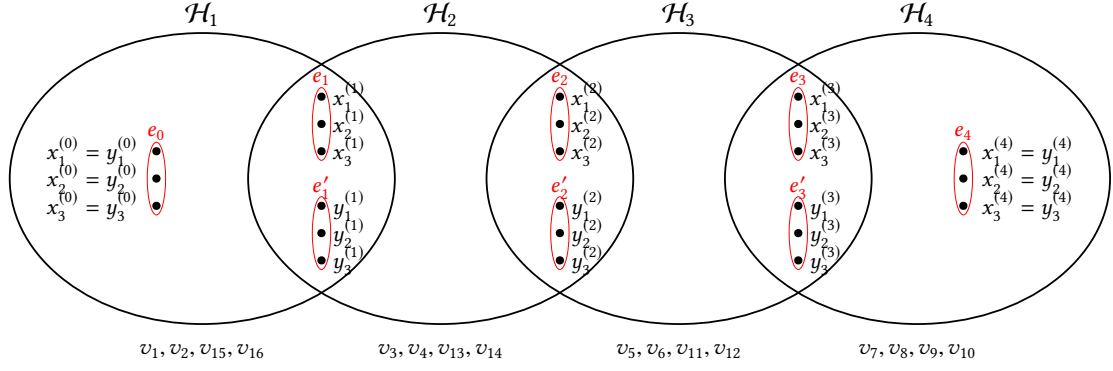}
	}%
	\hfill
	\subfloat[Edge sequence of the auxiliary $3$-uniform tight cycle.]{%
		\includegraphics[clip,width=.92\textwidth]{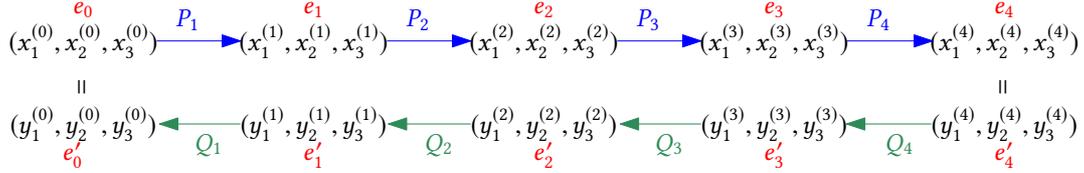}
	}%
	\hfill
	\subfloat[Vertex sequence of the resulting tight cycle.]{%
		\includegraphics[clip,width=0.95\textwidth]{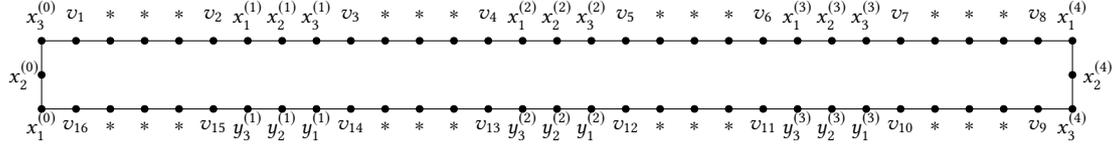}
	}%
	\caption{Finding a tight cycle in a path of blocks when $k=t=4$. In Figure (c), $*$ represents an internal vertex of a path $P_i$ or $Q_i$. %Figure (a) shows a path of blocks with their corresponding vertices in $B_t$ below. Figure (b) shows the edge sequence of the auxiliary tight cycle constructed in \cref{eq:edge-seq}. Finally, Figure (c) shows the resulting tight cycle in $H$ covering the relevant part of $B_t$.
	}
	\label{fig:blockcycles}
\end{figure}

\begin{step}{3 (Lift to tight cycles)}
  This step is the crucial part of the argument. To make it easier to follow the proof, \cref{fig:blockcycles} provides an example for $k=t=4$.

  We will find in each path of blocks an auxiliary tight cycle in $ K^{(k-1)}(B_1, \ldots, B_{k-1}) $ of the desired form to apply \cref{obs:liftcycle}. Let $\bP= (\cH_1, \ldots, \cH_{t})$ be one of the paths. Choose disjoint edges $e_0 = \left\{ x_1^{(0)}, \ldots, x_{k-1}^{(0)} \right\} \in E(\cH_1)$ and $e_t = \left\{ x_1^{(t)}, \ldots, x_{k-1}^{(t)} \right\} \in E(\cH_{t})$. For each $s \in [t - 1]$, further choose two edges $e_s = \left\{ x_1^{(s)}, \ldots, x_{k-1}^{(s)} \right\} \in E(\cH_s) \cap E(\cH_{s+1})$ and $e'_s = \left\{ y_1^{(s)}, \ldots, y_{k-1}^{(s)} \right\} \in E(\cH_s) \cap E(\cH_{s+1})$ so that all chosen edges are pairwise disjoint. We identify $x_i^{(0)} = y_i^{(0)}$ and $x_i^{(s)} = y_i^{(s)}$ for every $i \in [k-1]$, and $e_0=e'_0$ and $e_{t}=e'_{t}$. Assume without loss of generality, that $ x_i^{(s)} \in B_i $ for every $ i \in [k-1]$ and all $s = 0, \ldots, t$.

  By construction, every block $\cH$ has the property guaranteed in \cref{lem:shortpaths}. Therefore, for every $s \in [t]$, there is a tight path $P_s \subset \cH_s$ of length $2k-3$ which starts at $(x_2^{(s-1)}, \ldots, x_{k-1}^{(s-1)})$, ends at $(x_1^{(s)}, \ldots, x_{k-2}^{(s)})$ and (internally) avoids all previously used vertices.\footnote{Note that the number of previously used vertices in $V_j$ is at most $\gamma|V_j|$ for every $j \in [k-1]$ since every tight cycle in $G$ uses the same number of vertices from each part.}
  Similarly, there is for every $s \in [t]$ a tight path $Q_s \subset \cH_s$ of length $2k-3$ which starts at and $(y_1^{(s)}, \ldots, y_{k-2}^{(s)})$, ends at $(y_2^{(s-1)}, \ldots, y_{k-1}^{(s-1)})$ and (internally) avoids all previously used vertices.

  %For $v \in B_k$, we will write $v \in \cH_i$ if $H_v \in \cH_i$.
  Let $U \subset B_k$ be the set of vertices $v$ for which $H_v \in \cH_i$ for some $i \in [t]$. To finish the proof, we want to apply \cref{obs:liftcycle} to cover $U$.
  Label $U = \{v_1, \ldots, v_{4t}\}$ so that $H_{v_{2i+1}}, H_{v_{2i+2}}, H_{v_{4t-2i}}, H_{v_{4t-2i -1}}  \in \cH_i$ for all $i = 0, \ldots, t-1$.
  Consider now the tight cycle $\cC$ in $K^{k-1}\left( B_1, \ldots, B_{k-1} \right)$ with edge sequence
  \begin{equation}\label{eq:edge-seq}
  e'_0=e_0, P_1, e_1, P_2,e_2, \ldots, P_t, e_t = e'_t, Q_t, \ldots,e_1, Q_1, e'_0=e_0
  \end{equation}
  and relabel $V(\cC)$ so that it's vertex sequence is $(u_{1,1}, \ldots, u_{1,k-1}, \ldots, u_{t,1}, \ldots, u_{4t,k-1})$
  (i.e.\ $u_{1,i} = x_i^{(0)}$ for $i \in [k-1]$, $u_{2,1}, \ldots, u_{2,k-1}$ are the internal vertices of $P_1$\footnote{Note that all $P_i$ and $Q_i$ have $3k-5$ vertices and hence $k-1$ internal vertices.}, $u_{3,i} = x_i^{(1)}$ for all $i \in [3]$ and so on).
  By construction, $\cC$ has the desired properties to apply \cref{obs:liftcycle}, finishing the proof. Note that it is important here that every block $\cH_i$ has size $4$ since we cover $2$ vertices of every block ``going forwards'' and $2$ vertices ``going backwards''.
\end{step}
\vspace{-0.75cm}
\end{proof}

\subsection{Proof \texorpdfstring{of \cref{thm:main}}{the main theorem}.}\label{sec:proof}

  Fix $\alpha,r,n \in \N$ and a $k$-graph $G$ with $\alpha(G) \leq \alpha$.
  Choose constants $ 0 < \beta, \gamma, \e \ll \max\{\alpha,r,k\}^{-1}$ so that $\gamma = \gamma(r,\e)$ works for \cref{lem:keyk} and $\beta = \beta(\alpha,r)$ works for \cref{lem:crownk}. The proof proceeds in $\alpha$ steps (the initial $k-1$ steps are done at the same time). No effort is made to calculate the exact number of cycles we use, we only care that it is bounded (i.e.\ independent of $n$).

\begin{step}{1, \ldots, k-1}
  By \cref{lem:crownk}, there is some $B \subset [n]$ of size $\beta n$ with an absorber $A_{k-1} \subset [n]$. Partition $B$ into $k-1$ sets $B_1^{(k-1)},\ldots, B_{k-1}^{(k-1)}$ of equal sizes.
  By \cref{lem:greedy}, there is a bounded number of vertex-disjoint monochromatic tight cycles in $[n] \setminus (B \cup A_{k-1})$ so that the set $R_{k-1}$ of uncovered vertices in $[n] \setminus (B \cup A_{k-1})$ satisfies $|R_{k-1}| \leq \gamma |B_1^{(k-1)}|$.
  Let $R'_{k-1} \subset R_{k-1}$ be the set of vertices $v$ with $|\Lk(v;B_1^{(k-1)}, \ldots, B_{k-1}^{(k-1)})| < \e |B_1^{(k-1)}| \cdots |B_{k-1}^{(k-1)}|$ and let $R_{k-1}'' = R_{k-1} \setminus R_{k-1}'$.
  By \cref{lem:keyk} we can find a bounded number of vertex-disjoint cycles in $B_1^{(k-1)} \cup \ldots \cup B_{k-1}^{(k-1)} \cup R_{k-1}''$ covering $R_{k-1}''$. Remove them and let $B_i^{(k)} \subset B_i^{(k-1)}$ be the set of leftover vertices for every $ i \in [k-1]$.
	\end{step}

  \begin{step}{j $(j = k, \ldots, \alpha)$} In each step $j$, we will define disjoint sets $B_1^{(j+1)}, \ldots, B_j^{(j+1)}, R_{j+1}',A_{j}$. Fix some $j \in \{k, \ldots, \alpha\}$ now and suppose we have built disjoint sets $B_1^{(j)}, \ldots, B_{j-1}^{(j)}, R_j'$ and absorbers $A_2, \ldots, A_{j-1}$.  By \cref{lem:crownk} there is some $B_j^{(j)} \subset R_j'$ of size $\beta |R_j'|$ with an absorber $A_j \subset R_j'$. By \cref{lem:greedy}, there is a bounded number of monochromatic tight cycles in $R_j' \setminus (A_j \cup B_j^{(j)})$ so that the set $R_{j+1}$ of uncovered vertices in $R_j' \setminus (A_j \cup B_j^{(j)})$ satisfies $|R_{j+1}| \leq \gamma |B_j^{(j)}|$. Let $R'_{j+1} \subset R_{j+1}$ be the set of vertices $v$ with $|\Lk(v;B_{t_1}^{(j)},\ldots, B_{t_{k-1}}^{(j)})| < \e \card{B_{t_1}^{(j)}} \cdots \card{B_{t_{k-1}}^{(j)}}$ for all $1 \leq t_1 < \ldots < t_{k-1} \leq j$ and let $R_{j+1}'' = R_{j+1} \setminus R_{j+1}'$. By ($\binom{j}{k}$ applications of) \cref{lem:keyk} we can find a bounded number of vertex-disjoint cycles in $B_1^{(j)} \cup \ldots \cup B_{j}^{(j)} \cup R_j''$ covering $R_j''$. Remove them and let $B_i^{(j+1)} \subset B_i^{(j)}$ be the set of leftover vertices for every $ i \in [j]$.
	\end{step}

  In the end we are left with disjoint sets $ B_1:=B_1^{(\alpha+1)}, \ldots,  B_\alpha:=B_\alpha^{(\alpha+1)},  B_{\alpha+1}:=R_{\alpha+1}' $ and corresponding absorbers $A_{k-1}, \ldots, A_\alpha$ ($A_{k-1}$ absorbs $B_1^{(\alpha+1)}, \ldots, B_{k-1}^{(\alpha+ 1)}$). All other vertices are covered by a bounded number of cycles.

  We will show now that	$R_{\alpha+1}' = \emptyset$, which finishes the proof. In order to do so, we assume the contrary and find an independent set of size $\alpha+1$.
  Note that every vertex in $B_j^{(j)} \setminus B_j$ must be part of a tight cycle of our disjoint collection of tight cycles with one part in $R_{j+1}$ and hence $\card{B_j^{(j)} \setminus B_j} \leq \card{R_{j+1}} \leq \gamma \card{B_j^{(j)}}$. It follows that $\card{B_j} \geq (1- \gamma) \card{B_j^{(i)}}$ for every $1 \leq j \leq i \leq \alpha$ and we conclude
  \begin{align*}
    \Lk\left( v;B_{i_1},\ldots, B_{i_{k-1}} \right) &\leq \Lk\left( v;B_{i_1}^{(i-1)},\ldots, B_{i_{k-1}}^{(i-1)} \right) \\
    &\leq \e \card{B_{i_1}^{(i-1)}} \cdots \card{B_{i_{k-1}}^{(i-1)}}\\
    &\leq \e (1-\gamma)^{-(k-1)}\card{B_{i_1}} \cdots \card{B_{i_{k-1}}}\\
    &\leq 2\e \card{B_{i_1}} \cdots \card{B_{i_{k-1}}}
  \end{align*}
  for every $i \in \left\{ k, \ldots, \alpha+1 \right\}$, $1 \leq i_1 < \ldots < i_{k-1} < i$ and $ v \in B_i$.
  By the following lemma, there is an independent set of size $\alpha+1$, a contradiction.\qed

\begin{lemma}
  For all $k,m \in \N$ there is some $ \e = \e(k,m)>0$ such that the following is true for every $k$-graph $ H $ and all non-empty, disjoint sets $B_1, \ldots,B_m \subset V(H)$. If $\card{\Lk(v;B_{i_1}, \ldots, B_{i_{k-1}})} \leq \e \card{B_{i_1}} \cdots \card{B_{i_{k-1}}}$ for all $i \in \left\{ k, \ldots, m \right\}$, $1 \leq i_1 < \ldots < i_{k-1} < i$ and $ v \in B_i$, then there is an independent transversal, i.e.\ an independent set $\left\{ v_1, \ldots, v_m \right\}$ with $v_i \in B_i$ for all $i \in [m]$.
  \label{lem:indtrans}
\end{lemma}
We will prove the lemma for $\e(k,m) = m^{-(k-1)^{2}}$.
\begin{proof}
  Let $\delta=m^{-(k-1)}$ and $\e = \delta^{k-1}$. Choose $v_m \in B_m$ arbitrarily and assume now that $v_m, \ldots, v_{j+1}$ are chosen for some $j \in [m-1]$.
  Given $s \in \left\{ 2, \ldots, k-1 \right\}$ and $\mathbf i = (i_1, \ldots, i_k)$ with $1 \leq i_1 < \ldots <i_{s-1} < i_s = j < i_{s+1} < \ldots < i_k \leq m$, define
  \[\overline B_j(s,\mathbf i) := \left\{ u \in B_j : \card{\Lk \left( v_{i_k}, \ldots, v_{i_{s+1}},u;B_{i_{s-1}}, \ldots, B_{i_1} \right)} \geq \e/\delta^{k-s}\card{B_{i_1}} \cdots \card{B_{i_{s-1}}} \right\}.\]
  Furthermore, given $ \mathbf i = (i_1, \ldots, i_k)$ with $j = i_1 < i_2 < \ldots < i_k \leq m$, define
  \[\overline B_j(1,\mathbf i):=\mathrm{N}\left(v_{i_k}, \ldots, v_{i_2};B_{i_1}\right),\]
the neighbourhood of $\{v_{i_2},\ldots,v_{i_k}\}$ in $B_{i_1}$.
  Note that, by choice of $v_{m}, \ldots, v_{j+1}$, we have $\card{\overline B_j(s,\mathbf i)} < \delta \card{B_j}$ for every $s \in \left\{ 2, \ldots, k-1 \right\}$ and $ \card{\overline B_j(1, \mathbf i)} < \e/\delta^{k-2} \card{B_j} = \delta \card{B_j}$.
  Since there are at most $\binom{m-1}{k-1} < 1/\delta$ choices for $(s,\mathbf i)$, we can choose some $v_j \in B_j \setminus \bigcup_{s,\mathbf i} \overline B_j(s,\mathbf i)$. Clearly, at the end of this process, $\left\{ v_1, \ldots, v_m \right\}$ will be independent.
\end{proof}

\providecommand{\bysame}{\leavevmode\hbox to3em{\hrulefill}\thinspace}
\providecommand{\MR}{\relax\ifhmode\unskip\space\fi MR }
\providecommand{\MRhref}[2]{%
	\href{http://www.ams.org/mathscinet-getitem?mr=#1}{#2}
}
\providecommand{\href}[2]{#2}

%\bibliography{./bib}
%\bibliographystyle{./amsplain-abbr}
\end{document}